\documentclass{llncs}

\usepackage{amsmath} 
\usepackage{amsbsy} 
\usepackage{amssymb}
\usepackage{bm}
\usepackage{cite}
\usepackage{graphicx}

\title{A Singularly Perturbed Boundary Value Problems with Fractional Powers of Elliptic Operators}

\author{Petr N. Vabishchevich$^{1,2}$}

\institute{$^1$Nuclear Safety Institute, 52, B. Tulskaya, 115191 Moscow, Russia \\
$^2$North-Eastern Federal University, 58, Belinskogo, 677000 Yakutsk, Russia}

\begin{document}

\maketitle

\begin{abstract}
A boundary value problem for a fractional power $0 < \varepsilon < 1$ of the second-order elliptic operator is considered.
The boundary value problem is singularly perturbed when $\varepsilon \rightarrow 0$.
It is solved numerically using a time-dependent problem for a pseudo-parabolic equation.
For the auxiliary Cauchy problem, the standard two-level schemes with weights are applied.
The numerical results are presented for a model two-dimen\-sional
boundary value problem with a fractional power of an elliptic operator.
Our work focuses on the solution of the boundary value problem with $0 < \varepsilon \ll 1$.
\end{abstract}

\section{Introduction}

Non-local applied mathematical models based on the use of fractional derivatives in time and space 
are actively discussed in the literature \cite{baleanu2012fractional,kilbas2006theory}. 
Many models, which are used in applied physics, biology, hydrology, and finance,
involve both sub-diffusion (fractional in time) and super-diffusion (fractional in space) operators. 
Super-diffusion problems are treated as problems with a fractional power of an elliptic operator.
For example, suppose that in a bounded domain $\Omega$ on the set of functions 
$u(\bm x) = 0, \ \bm x \in \partial \Omega$, 
there is defined the operator $\mathcal{A}$: $\mathcal{A}  u = - \triangle u, \ \bm x \in \Omega$. 
We seek the solution of the problem for the equation with the fractional power of an elliptic operator:
\[
 \mathcal{A}^{\varepsilon } u = f,
\] 
with $0 < \varepsilon < 1$ for a given $f(\bm x), \ \bm x \in \Omega$. 

To solve problems with the fractional power of an elliptic operator, we can apply
finite volume or finite element methods oriented to using arbitrary
domains discretized by irregular computational grids \cite{KnabnerAngermann2003,QuarteroniValli1994}.
The computational realization is associated with the implementation of the matrix function-vector multiplication.
For such problems, different approaches \cite{higham2008functions} are available.
Problems of using Krylov subspace methods with the Lanczos approximation
when solving systems of linear equations associated with
the fractional elliptic equations are discussed, e.g., in \cite{ilic2009numerical}.
A comparative analysis of the contour integral method, the extended Krylov subspace method, and the preassigned 
poles and interpolation nodes method for solving
space-fractional reaction-diffusion equations is presented in \cite{burrage2012efficient}.
The simplest variant is associated with the explicit construction of the solution using the known
eigenvalues and eigenfunctions of the elliptic operator with diagonalization of the corresponding matrix
\cite{bueno2012fourier,ilic2005numerical,ilic2006numerical}. 
Unfortunately, all these approaches demonstrates too high computational complexity for multidimensional problems.

We have proposed \cite{vabishchevich2014numerical} a computational algorithm for solving
an equation with fractional powers of elliptic operators on the basis of
a transition to a pseudo-parabolic equation.
For the auxiliary Cauchy problem, the standard two-level schemes are applied.
The computational algorithm is simple for practical use, robust, and applicable to solving
a wide class of problems. A small number of pseudo-time steps is required to reach a steady-state solution.
This computational algorithm for solving equations with fractional powers of operators
is promising when considering transient problems. 

The boundary value problem for the fractional power of an elliptic operator
is singularly perturbed when $\varepsilon \rightarrow 0$. To solve it numerically,
we focus on numerical methods that are designed for classical elliptic problems
of convection-diffusion-reaction \cite{miller2012fitted,roos2008robust}.
In particular, the main features are taken into account via using locally refining grids.
The standard strategy of goal-oriented error control for conforming finite element 
discretizations \cite{ainsworth2011posteriori,bangerth2013adaptive} is applied.
 
\section{Problem formulation}

In a bounded polygonal domain $\Omega \subset R^m$, $m=1,2,3$ with the Lipschitz continuous boundary $\partial\Omega$,
we search the solution for the problem with a fractional power of an elliptic operator.
Define the elliptic operator as
\begin{equation}\label{1}
  \mathcal{A}  u = - {\rm div} ( k({\bm x}) {\rm grad} \, u )
\end{equation} 
with coefficient $0 < k_1 \leq k({\bm x}) \leq k_2$.
The operator $\mathcal{A}$ is defined on the set of functions $u({\bm x})$ that satisfy
on the boundary $\partial\Omega$ the following conditions:
\begin{equation}\label{2}
  u ({\bm x}) = 0,
  \quad {\bm x} \in \partial \Omega .
\end{equation} 

In the Hilbert space $H = L_2(\Omega)$, we define the scalar product and norm in the standard way:
\[
  (u,v) = \int_{\Omega} u({\bm x}) v({\bm x}) d{\bm x},
  \quad \|u\| = (u,u)^{1/2} .
\] 
For the spectral problem
\[
 \mathcal{A}  \varphi_k = \lambda_k \varphi_k, 
 \quad \bm x \in \Omega , 
\] 
\[
  \varphi_k ({\bm x}) = 0,
  \quad {\bm x} \in \partial \Omega , 
\] 
we have 
\[
 \lambda_1 \leq \lambda_2 \leq ... ,
\] 
and the eigenfunctions  $ \varphi_k, \ \|\varphi_k\| = 1, \ k = 1,2, ...  $ form a basis in $L_2(\Omega)$. Therefore, 
\[
 u = \sum_{k=1}^{\infty} (u,\varphi_k) \varphi_k .
\] 
Let the operator $\mathcal{A}$ be defined in the following domain:
\[
 D(\mathcal{A} ) = \{ u \ | \ u(\bm x) \in L_2(\Omega), 
 \quad  \sum_{k=0}^{\infty} | (u,\varphi_k) |^2 \lambda_k < \infty \} .
\] 
Under these conditions the operator $\mathcal{A}$ is self-adjoint and positive defined: 
\begin{equation}\label{3}
  \mathcal{A}  = \mathcal{A}^* \geq \delta I ,
  \quad \delta > 0 ,    
\end{equation} 
where $I$ is the identity operator in $H$. For $\delta$, we have $\delta = \lambda_1$.
In applications, the value of $\lambda_1$ is unknown (the spectral problem must be solved).
Therefore, we assume that $\delta \leq \lambda_1$ in (\ref{3}).
Let us assume for the fractional power of the  operator $\mathcal{A}$
\[
 \mathcal{A}^{\varepsilon} u =  \sum_{k=0}^{\infty} (u,\varphi_k) \lambda_k^{\varepsilon}  \varphi_k .
\] 
We seek the solution of the problem with the fractional power of the operator $\mathcal{A}$. 
The solution $u(\bm x)$ satisfies the equation
\begin{equation}\label{4}
 \mathcal{A}^{\varepsilon } u = f,
\end{equation} 
with $0 < \varepsilon < 1$ for a given $f(\bm x), \ \bm x \in \Omega$. 

The key issue in the study of the computational algorithm for solving the problem (\ref{4})
is to establish the stability of the approximate solution with respect to small perturbations of the 
right-hand side in various norms.
In view of (\ref{3}),  the solution of the problem (\ref{4}) satisfies the a priori estimate
\begin{equation}\label{5}
 \|u\| \leq  \delta^{-\varepsilon} \|f\| ,   
\end{equation}
which is valid for all $0 < \varepsilon  < 1$. 

The boundary value problem for the fractional power of the elliptic operator (\ref{4}) 
demonstrates a reduced smoothness when $\varepsilon \rightarrow 0$.
For the solution, we have (see, e.g., \cite{yagi2009abstract}) the estimate
\[
 \|u\|_{2\varepsilon} \leq C \|f\|,
\] 
with $0 \leq \varepsilon < 1/2$, is $\|\cdot\|_{2\varepsilon}$ is the norm in $H^{2\varepsilon} (\Omega)$.
For the limiting solution, we have 
\[
 u_0(\bm x) = f(\bm x), 
 \quad \bm x \in \Omega .
\] 
Thus, a singular behavior of the solution of the  problem (\ref{4}) appears with
$\varepsilon \rightarrow 0$ and is governed by the right-hand side $f(\bm x)$.

\section{Discretization in space}

To solve numerically the problem (\ref{4}), we employ finite-element 
approximations in space \cite{brenner2008mathematical,Thomee2006}. 
For (\ref{1}) and (\ref{2}), we define the bilinear form
\[
 a(u,v) = \int_{\Omega } k \, {\rm grad} \, u \, {\rm grad} \, v .
\] 
By (\ref{3}), we have
\[
 a(u,u) \geq \delta \|u\|^2 .  
\]
Define a subspace of finite elements $V^h \subset H^1_0(\Omega)$.
Let $\bm x_i, \ i = 1,2, ..., M_h$ be triangulation points for the domain $\Omega$.
Define pyramid function $\chi_i(\bm x) \subset V^h, \ i = 1,2, ..., M_h$, where
\[
 \chi_i(\bm x_j) = \left \{
 \begin{array}{ll} 
 1, & \mathrm{if~}  i = j, \\
 0, & \mathrm{if~}  i \neq  j .
 \end{array}
 \right . 
\] 
For $v \in V_h$, we have
\[
 v(\bm x) = \sum_{i=i}^{M_h} v_i \chi_i(\bm x),
\] 
where $v_i = v(\bm x_i), \ i = 1,2, ..., M_h$.
We have defined  Lagrangian finite elements of first degree, i.e., based on the piecewise-linear
approximation. We will also use Lagrangian finite elements of second degree defined in a similar way.

We define the discrete elliptic operator $A$ as
\[
 (A y, v) = a(y,v),
 \quad \forall \ y,v \in V^h . 
\]
The fractional power of the operator $A$ is defined similarly to $\mathcal{A}^{\varepsilon}$.
For the spectral problem
\[
 A \widetilde{\varphi}_k = \widetilde{\lambda}_k 
\] 
we have
\[
 \widetilde{\lambda}_1 \leq \widetilde{\lambda}_2 \leq ... \leq  \widetilde{\lambda}_{M_h},
 \quad \| \widetilde{\varphi}_k\| = 1,
 \quad k = 1,2, ..., M_h . 
\]
The domain of definition for the operator $A$ is
\[
 D(A) = \{ y \ | \ y \in V^h, 
 \quad  \sum_{k=0}^{M_h} | (y,\widetilde{\varphi}_k) |^2 \widetilde{\lambda}_k < \infty \} .
\]
The operator $A$ acts on a finite dimensional space $V^h$ defined on  the domain $D(A)$
and, similarly to (\ref{3}), 
\begin{equation}\label{6}
 A = A^* \geq \delta I ,
 \quad \delta > 0 , 
\end{equation} 
where $\delta \leq \lambda_1 \leq \widetilde{\lambda}_1$. 
For the fractional power of the operator $A$, we suppose
\[
 A^{\varepsilon} y = \sum_{k=1}^{M_h} (y, \widetilde{\varphi}_k) \widetilde{\lambda}_k^{\varepsilon} 
 \widetilde{\varphi}_k .
\] 
For the problem (\ref{4}), we put into the correspondence the operator equation 
for $w(t) \in V^h$:
\begin{equation}\label{7}
 A^{\varepsilon } w = \psi, 
\end{equation} 
where $\psi = P f$ with $P$ denoting $L_2$-projection onto $V^h$.
For the solution of the problem (\ref{6}), (\ref{7}), we obtain (see (\ref{5}))
the estimate
\begin{equation}\label{8}
 \|w\| \leq  \delta^{-\varepsilon} \|\psi\| ,   
\end{equation}
for all $0 < \varepsilon  < 1$. 

\section{Singularly perturbed problem for a diffusion-reaction equation} 

The object of our study is associated with the development of a computational algorithm
for approximate solving the singularly perturbed problem (\ref{4}).
After constructing a finite element approximation, we arrive at equation (\ref{7}).
Features of the solution related to a boundary layer are investigated on
a model singularly perturbed problem for an equation of diffusion-reaction. 
The key moment is associated with selecting adaptive computational grids (triangulations).

In view of
\[
 A^\varepsilon = \left ( \exp (\ln A) \right )^\varepsilon  =
 I + \varepsilon  \ln A + \mathcal{O} (\varepsilon^2 ) ,
\]
we put the problem (\ref{7}) into the correspondence with solving the equation
\begin{equation}\label{9}
 \varepsilon A u + u = \psi .
\end{equation}
The equation (\ref{9}) corresponds to solving the Dirichlet problem (see the condition (\ref{2}))
for the diffusion-reaction equation
\begin{equation}\label{10}
 - \varepsilon \, {\rm div}  (k({\bm x}) {\rm grad} \, u) + u = f(\bm x),
 \quad \bm x \in \Omega . 
\end{equation}
Basic computational algorithms for the singularly perturbed boundary
problem (\ref{2}), (\ref{10}) are considered, for example, in \cite{miller2012fitted,roos2008robust}. 

In terms of practical applications, the most interesting approach is based on an adaptation of
a computational grid to peculiarities of the problem solution via a posteriori error estimates.
Among main approaches, we highlight the strategy of the goal-oriented error control for conforming finite element
discretizations \cite{ainsworth2011posteriori,bangerth2013adaptive}, which is applied to approximate
solving boundary value problems for elliptic equations.

The strategy of goal-oriented error control is based on choosing a calculated functional. 
The accuracy of its evaluation is tracked during computations. In our Dirichlet problem for
the second-order elliptic equation, the solution is varied drastically near the boundary.
So, it seems natural to control the accuracy of calculations for the normal derivatives of the solution
(fluxes) across the boundary or a portion of it. Because of this, we put
\[
 G(u) = - \int_{\partial \Omega } \varepsilon k(\bm x) ({\rm grad} \, u \cdot \bm n ) d \, \bm x,
\] 
where $\bm n$ is the outward normal to the boundary.
An adaptation of a finite element mesh is based on an iterative local refinement of the grid 
in order to evaluate the goal functional with a given accuracy $\eta$
on the deriving approximate solution $u_h$, i.e.,
\[
 | G(u) - G(u_h) | \leq \eta .
\]
To conduct our calculations, we used the \textsf{FEniCS} framework (see, e.g., \cite{logg2012automated})
developed for general engineering and scientific calculations via finite elements.
Features of the goal-oriented procedure for local refinement of the computational grid
are described in \cite{rognes2013automated} in detain. Here, we consider only a key idea of 
the adaptation strategy of finite element meshes, which is associated with selecting the goal functional.

The model problem (\ref{2}), (\ref{10}) is considered with
\[
 k(\bm x) = 1,
 \quad f(\bm x) = (1-x_1) x^2_2,
\] 
in the unit square ($\Omega = (0,1) \times (0,1)$).
The threshold of accuracy for calculating the functional $G(u)$ is defined by the value of $\eta = 10^{-5}$. 
As an initial mesh, there is used the uniform grid obtained via division by 8 intervals in each direction
(step 0 --- 128 cells).

First, Lagrangian finite elements of first order have been used in our calculations. For this case,
the improvement of the goal functional during the iterative procedure of adaptation is illustrated by the data 
presented in Table~\ref{tab-1}. Table~\ref{tab-2} demonstrates values of the goal functional
$G(u_h)$ calculated on the final computational grid, the number of vertices of this final grid and 
the number of adaptation steps for solving the problem at various values of the small parameter $\varepsilon$.
These numerical results demonstrate the efficiency of the proposed strategy for goal-oriented error control 
for conforming finite element discretizations applied to approximate solving singular perturbed problems 
of diffusion-reaction (\ref{2}), (\ref{10}).

\begin{table}
\begin{center}
 \caption{Calculation of the goal functional during adaptation steps}
 \begin{tabular}{c|lr|lr|lr}\label{tab-1}
  $\varepsilon$ & \multicolumn{2}{|c}{$10^{-1}$} & \multicolumn{2}{|c}{$10^{-3}$} & \multicolumn{2}{|c}{$10^{-5}$} \\
  \hline
  Step of adaptation $s$ & ~~ $G(u_h)$  & ~ $M_h $  ~ & ~~ $G(u_h)$  & ~ $M_h $ ~ & ~~ $G(u_h)$  & ~ $M_h $ ~  \\
  \hline
   0  & ~~ 0.087608 & 81    & ~~ 0.0056973 & 81    & ~~ 0.00006643  & 81 \\
   1  & ~~ 0.110432 & 97    & ~~ 0.0107507 & 98    & ~~ 0.00015584  & 95 \\
   2  & ~~ 0.116155 & 140   & ~~ 0.0129506 & 132   & ~~ 0.00023996  & 120 \\
   3  & ~~ 0.119766 & 222   & ~~ 0.0155597 & 195   & ~~ 0.00035644  & 164 \\
   4  & ~~ 0.122702 & 384   & ~~ 0.0175113 & 305   & ~~ 0.00050472  & 225 \\
   5  & ~~ 0.125653 & 694   & ~~ 0.0194985 & 466   & ~~ 0.00068154  & 349 \\
   6  & ~~ 0.127950 & 1235  & ~~ 0.0210232 & 754   & ~~ 0.00090839  & 550 \\
   7  & ~~ 0.128835 & 2179  & ~~ 0.0221562 & 1279  & ~~ 0.00115091  & 853 \\
   8  & ~~ 0.129542 & 3841  & ~~ 0.0229284 & 2132  & ~~ 0.00137740  & 1242 \\
   9  & ~~ 0.129940 & 6540  & ~~ 0.0234492 & 3753  & ~~ 0.00161273  & 1865 \\
   10 & ~~ 0.130149 & 11040 & ~~ 0.0237487 & 6626  & ~~ 0.00181249  & 2711 \\
\end{tabular}
\end{center} 
\end{table}

\begin{table}
\begin{center}
 \caption{Adaptation for various values of $\varepsilon$}
 \begin{tabular}{c|l|c|c}\label{tab-2}
  $\varepsilon$ & ~ Goal functional $G(u_h)$ ~ & ~ Number of vertices ~ & ~ Number of adaptation steps $s$ \\
  \hline
  $10^{-1}$  & ~~~~~~~ 0.130396     & 51868  & 13 \\
  $10^{-2}$  & ~~~~~~~ 0.064867     & 72297  & 14 \\
  $10^{-3}$  & ~~~~~~~ 0.024191     & 90170  & 15 \\
  $10^{-4}$  & ~~~~~~~ 0.008061     & 67476  & 16 \\
  $10^{-5}$  & ~~~~~~~ 0.002580     & 99003  & 18 \\
\end{tabular}
\end{center} 
\end{table} 

Next, similar results have been obtained using Lagrangian finite elements of
second order. For this case, summary data are presented in Table~\ref{tab-3}.
As expected, the desired accuracy $\eta = 10^{-5}$ is reached on adaptive
meshes of smaller sizes than in the case of Lagrangian finite elements of
first order (see Table~\ref{tab-2} for a comparison).

\begin{table}
\begin{center}
 \caption{Adaptation for Lagrangian elements of second order}
 \begin{tabular}{c|l|c|c}\label{tab-3}
  $\varepsilon$ & ~ Goal functional $G(u_h)$ ~ & ~ Number of vertices ~ & ~ Number of adaptation steps $s$  \\
  \hline
  $10^{-1}$  & ~~~~~~~ 0.130423     & 3574  & 7 \\
  $10^{-2}$  & ~~~~~~~ 0.064884     & 5137  & 8 \\
  $10^{-3}$  & ~~~~~~~ 0.024184     & 6573  & 9 \\
  $10^{-4}$  & ~~~~~~~ 0.008076    & 12775  & 11 \\
  $10^{-5}$  & ~~~~~~~ 0.002574    & 18501  & 12 \\
\end{tabular}
\end{center} 
\end{table} 

\section{Numerical algorithm for the problem with a fractional power}

An approximate solution of the problem (\ref{7}) is sought as a solution of an auxiliary pseudo-time 
evolutionary problem \cite{vabishchevich2014numerical}.
Assume that
\[
 y(t) = \delta^{\varepsilon} (t (A - \delta I) + \delta I)^{-\varepsilon} y(0) .
\]
Therefore
\[
 y(1) =  \delta^{\varepsilon} A^{-\varepsilon} y(0)
\]  
and then $w = y(1)$.
The function $y(t)$ satisfies the evolutionary equation
\begin{equation}\label{11}
 (t D + \delta I) \frac{d y}{d t} + \varepsilon D y = 0 ,
 \quad 0 < t \leq 1 ,
\end{equation}  
where
\[
 D = A - \delta I .
\] 
By (\ref{6}), we get
\begin{equation}\label{12}
 D = D^* > 0 .
\end{equation} 
We supplement (\ref{11}) with the initial condition
\begin{equation}\label{13}
 y(0) = \delta^{-\varepsilon} \psi .  
\end{equation} 
The solution of equation (\ref{7}) can be defined as the solution of the Cauchy problem 
(\ref{11})--(\ref{13}) at the final pseudo-time moment $t=1$.

For the solution of the problem (\ref{11}), (\ref{13}), it is possible to obtain various a priori estimates.
The elementary estimate that is consistent with the estimate (\ref{8}) have the form
\begin{equation}\label{14}
 \|y(t)\| \leq \|y(0)\| .
\end{equation} 
To get (\ref{14}), multiply scalarly equation (\ref{11}) by $\varepsilon y + t dy/ dt$.

To solve numerically the problem (\ref{11}), (\ref{13}),
we use the simplest implicit two-level scheme.
Let $\tau$ be a step of a uniform grid in time such that $y^n = y(t^n), \ t^n = n \tau$, $n = 0,1, ..., N, \ N\tau = 1$.
Let us approximate equation (\ref{11}) by the implicit two-level scheme
\begin{equation}\label{15}
 (t^{\sigma(n)} D + \delta I) \frac{ y^{n+1} - y^{n}}{\tau }
 + \varepsilon D y^{\sigma(n)} = 0,  \quad n = 0,1, ..., N-1,
\end{equation}
\begin{equation}\label{16}
 y^0 = \delta^{-\varepsilon} \psi  .
\end{equation} 
We use the notation
\[
  t^{\sigma(n)} = \sigma t^{n+1} + (1-\sigma) t^{n},
  \quad y^{\sigma(n)} = \sigma y^{n+1} + (1-\sigma) y^{n}.
\]
For $\sigma =0.5$, the difference scheme (\ref{15}), (\ref{16}) approximates the problem 
(\ref{11}), (\ref{12})
with the second order by $\tau$, whereas for other values of $\sigma$, we have only the first order.

\begin{theorem}\label{t-1}
For $\sigma \geq 0.5$ the difference scheme (\ref{15}), (\ref{16}) 
is unconditionally stable with respect to the initial data.
The approximate solution satisfies the estimate 
\begin{equation}\label{17}
 \|y^{n+1}\| \leq \|y^0\| , 
 \quad n = 0,1, ..., N-1.
\end{equation} 
\end{theorem} 
 
\begin{proof}
Rewrite equation (\ref{15}) in the following form:
\[
 \delta \frac{ y^{n+1} - y^{n}}{\tau } + D \left( \varepsilon y^{\sigma(n)} + t^{\sigma(n)}\frac{ y^{n+1} - y^{n}}{\tau } \right ) = 0 . 
\] 
Multiplying scalarly it by
\[
 \varepsilon y^{\sigma(n)} + t^{\sigma(n)}\frac{ y^{n+1} - y^{n}}{\tau },
\] 
in view of (\ref{12}), we arrive at
\[
 \left ( \frac{ y^{n+1} - y^{n}}{\tau }, y^{\sigma(n)} \right ) \leq 0 . 
\]
We have
\[
 y^{\sigma(n)} = \left( \sigma - \frac{1}{2} \right ) \tau \frac{y^{n+1} - y^{n}}{\tau }   +
 \frac{1}{2} (y^{n+1} + y^{n}) .
\]  
If $\sigma \geq 0.5$, then
\[
 \|y^{n+1}\| \leq \|y^n\| , 
 \quad n = 0,1, ..., N-1 .  
\] 
Thus, we obtain (\ref{17}).
\qquad\end{proof} 

The key point in approximate solving singularly perturbed boundary value problems is associated with
mesh adaptation. In the case of solving the problem (\ref{4}), we use finite element
approximations and proceed to the problem (\ref{7}) and then formulate the Cauchy problem (\ref{11}), (\ref{13}) 
approximated by the scheme (\ref{15}), (\ref{16}).
In our case, singularity is associated only with spatial variables.

The decomposition of the solution of the problem (\ref{11}), (\ref{13}) by eigenfunctions of the operator $A$
results in
\[
 y(t) = \sum_{k=1}^{N_h} a_k(t) \widetilde{\varphi}_k .
\] 
For coefficients $a_k(t)$, we get
\[
 a_k(t) = (\psi, \widetilde{\varphi}_k) (\delta + (\widetilde{\lambda}_k-\delta) t)^{-\varepsilon} ,
 \quad k = 1,2, ..., M_h . 
\]
Because of this, errors in specifying the initial conditions monotonically decrease for increasing $t$.
A similar behavior demonstrates an approximate solution of the Cauchy problem (\ref{11}), (\ref{13})
obtained using the fully implicit scheme with $\sigma = 1$ in (\ref{15}), (\ref{16}).
For the Crank-Nicolson scheme (i.e., $\sigma = 0.5$ in (\ref{15}), (\ref{16})), we cannot guarantee a monotone
decrease of errors in time, but the error at $t=1$ will not be more than at $t=0$.
The practical significance of such an analysis is that it provides us a simple adaptation strategy
for computational grids in solving the problem (\ref{11}), (\ref{13}), namely, 
spatial mesh adaptation is conducted at the first time step of calculations.

\section{Solution of a model problem}

Below, there are presented some results of numerical solving the problem (\ref{7}) 
for small values of $\varepsilon$. A computational algorithm must track a
singular behavior of the solution, which is directly related to the singular behavior of the right-hand side
$f(\bm x)$. Let us consider the problem (\ref{2}), (\ref{10}) in the unit square $\Omega = (0,1) \times (0,1)$
with
\[
 k(\bm x) = 1,
 \quad f(\bm x) = \left (1-x_1-\exp \left (-\frac{x_1}{\mu} \right)  \right ) \left (x_2^2-\exp \left (-\frac{1-x_2}{\mu} \right)  \right ) .
\]
The singularity of the right-hand side (the singularity of a numerical solution of the problem with a fractional power 
of an elliptic operator) results from existing a boundary layer at low values of $\mu$.

An adaptation of the computational grid is performed during the calculation of the first time step using 
the two-level scheme (\ref{15}), (\ref{16}). For the basic variant, it is assumed that
$\varepsilon = 10^{-2}$, $\mu = 10^{-2}$, the initial uniform spatial grid contains 8 intervals in each direction and
the time step is $\tau = 10^{-2}$. The parameter $\delta = 2 \pi^2$ corresponds the minimal eigenvalue 
of the elliptic operator $\mathcal{A}$. Mesh adaptation is carried out taking into account peculiarities 
of the right-hand side and the goal functional defined in the form
\[
 G(u; t = \tau) = - \int_{\partial \Omega } k(\bm x) ({\rm grad} \, u \cdot \bm n ) d \, \bm x .
\]
Next, the problem (\ref{2}), (\ref{10}) is solved using the derived grid in space and
the uniform grid in time. Thus, we apply the simplest one-stage starting adaptation of the computational grid
for numerical solving the unsteady problem. Lagrangian finite elements of second order are used.
For time-stepping, the  Crank-Nicolson ($\sigma = 0.5$ in (\ref{15})) scheme is utilized.
The sequence of calculated adaptive grids is shown in Fig.~\ref{fig-1}.
Note that this sequence is weakly dependent on the choice of a time step.
The goal functional dynamics for different levels of adaptation is presented
in Table~\ref{tab-4}. The problem is solved with different values of $\varepsilon$.

\begin{table}
\begin{center}
 \caption{Calculation of the goal functional during adaptation steps}
 \begin{tabular}{c|lr|lr|lr}\label{tab-4}
  $\varepsilon$ & \multicolumn{2}{|c}{$10^{-1}$} & \multicolumn{2}{|c}{$10^{-2}$} & \multicolumn{2}{|c}{$10^{-3}$} \\
  \hline
  Step of adaptation $s$ & ~~ $G(u_h; t=\tau)$  & ~ $M_h $  ~ & ~~ $G(u_h; t=\tau)$  & ~ $M_h $ ~ & ~~ $G(u_h; t=\tau)$  & ~ $M_h $ ~  \\
  \hline
   0  & ~~ 16.0955 & 289   & ~~ 21.8932 & 289    & ~~ 22.5773  & 289 \\
   1  & ~~ 24.3875 & 315   & ~~ 33.7810 & 315    & ~~ 34.9016  & 315 \\
   2  & ~~ 31.1692 & 399   & ~~ 43.8893 & 399    & ~~ 45.4226  & 399 \\
   3  & ~~ 37.0996 & 559   & ~~ 52.8249  & 559   & ~~ 54.7328  & 559 \\
   4  & ~~ 42.0854 & 833   & ~~ 60.4373 & 837    & ~~ 62.2786  & 834 \\
   5  & ~~ 45.7594 & 1270  & ~~ 66.2019 & 1282   & ~~ 68.4363  & 1264 \\
   6  & ~~ 48.6087 & 1849  & ~~ 70.4881 & 1885   & ~~ 73.2101  & 1889 \\
   7  & ~~ 50.3070 & 2753  & ~~ 73.1491 & 2778   & ~~ 75.9998  & 2774 \\
   8  & ~~ 51.2621 & 4067  & ~~ 74.6778 & 4120   & ~~ 77.5762  & 4125 \\
   9  & ~~ 51.9766 & 5965  & ~~ 75.8362 & 5968   & ~~ 78.7648  & 6028 \\
   10 & ~~ 52.2862 & 9201  & ~~ 76.3402 & 9235   & ~~ 79.2942  & 9261 \\
\end{tabular}
\end{center} 
\end{table}

\begin{figure}[htp]
\begin{minipage}{0.5\linewidth}
  \begin{center}
    \includegraphics[width=1.0\linewidth] {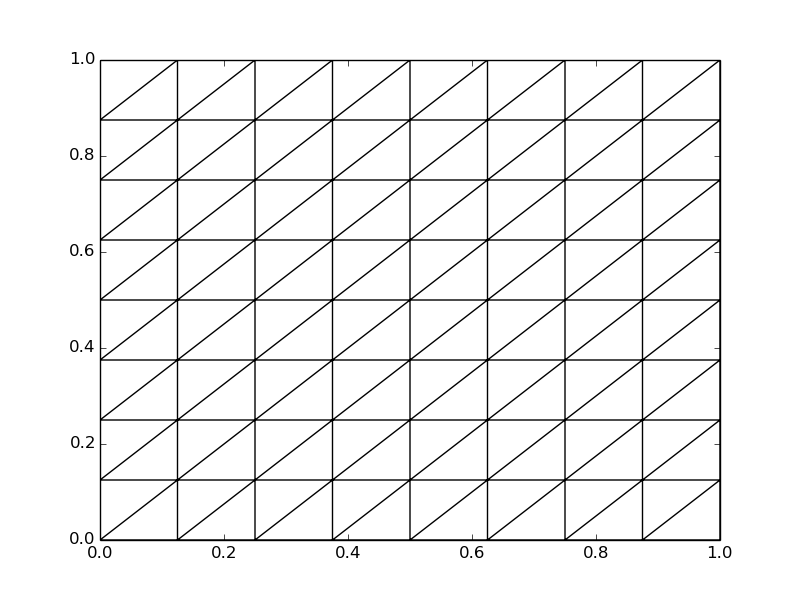} \\
      0 --- 128 cells
  \end{center}
\end{minipage}\hfill
\begin{minipage}{0.5\linewidth}
  \begin{center}
    \includegraphics[width=1.0\linewidth] {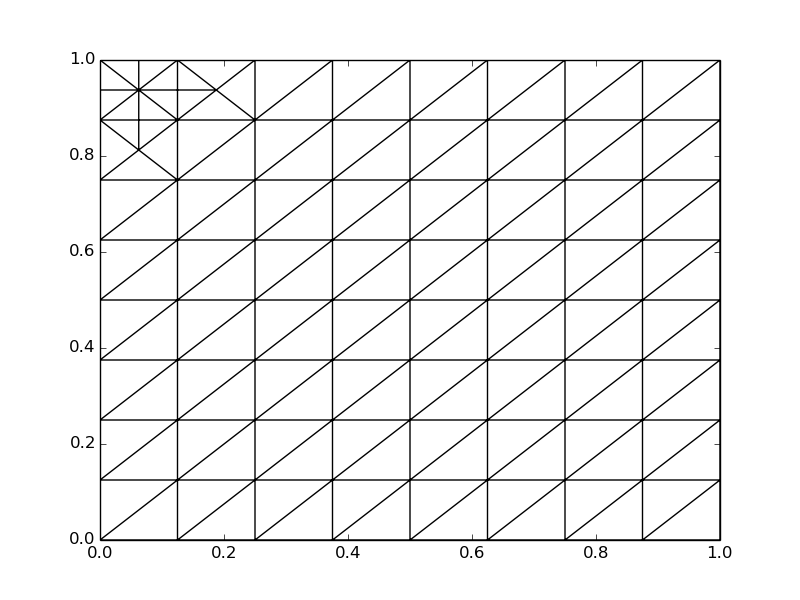} \\
    1 --- 140 cells
  \end{center}
\end{minipage}
\begin{minipage}{0.5\linewidth}
  \begin{center}
    \includegraphics[width=1.0\linewidth] {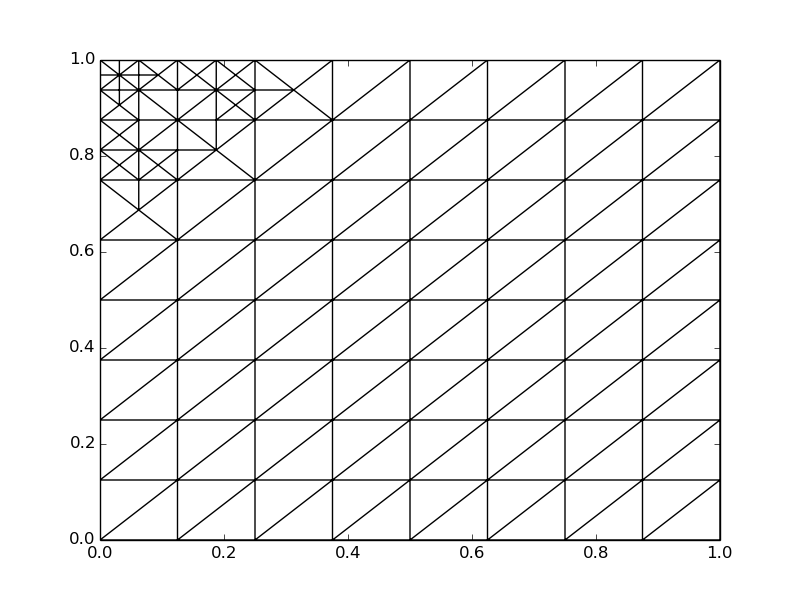} \\
    2 --- 180 cells
  \end{center}
\end{minipage}\hfill
\begin{minipage}{0.5\linewidth}
  \begin{center}
    \includegraphics[width=1.0\linewidth] {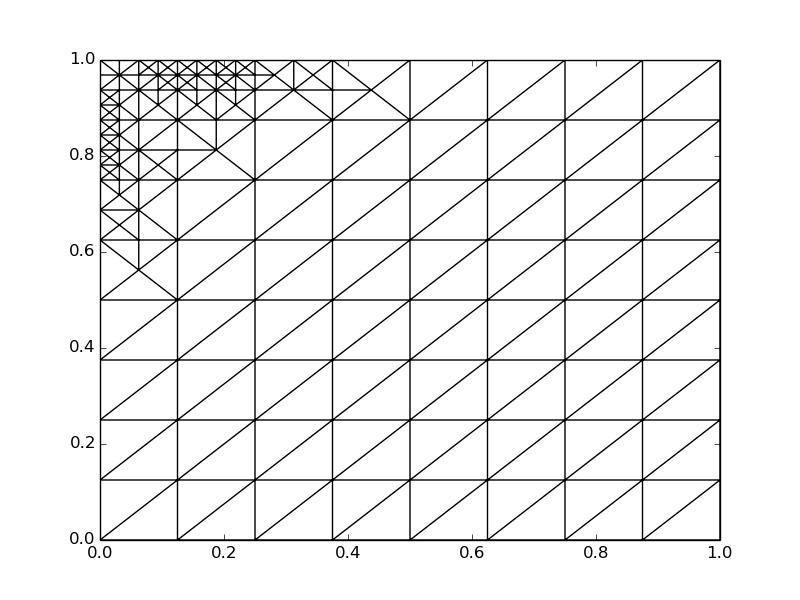} \\
      3 --- 256 cells
  \end{center}
\end{minipage}
\begin{minipage}{0.5\linewidth}
  \begin{center}
    \includegraphics[width=1.0\linewidth] {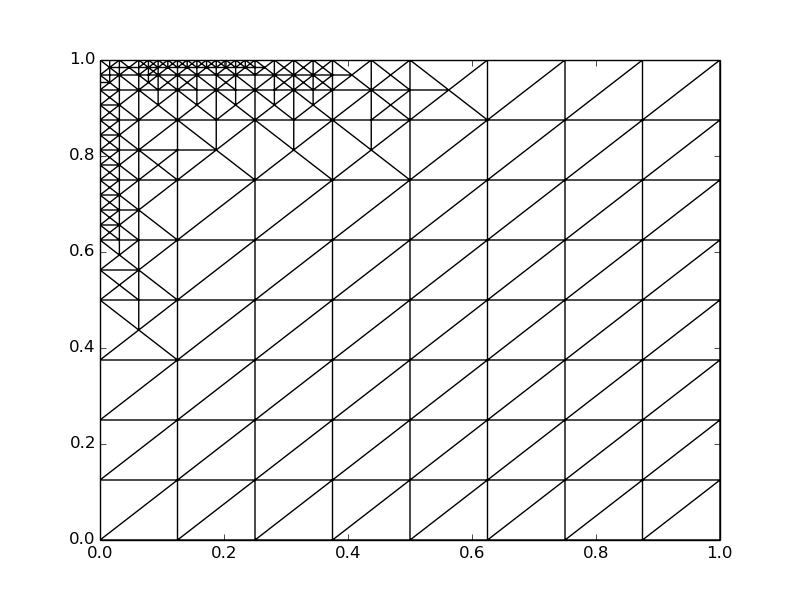} \\
    4 --- 388 cells
  \end{center}
\end{minipage}\hfill
\begin{minipage}{0.5\linewidth}
  \begin{center}
    \includegraphics[width=1.0\linewidth] {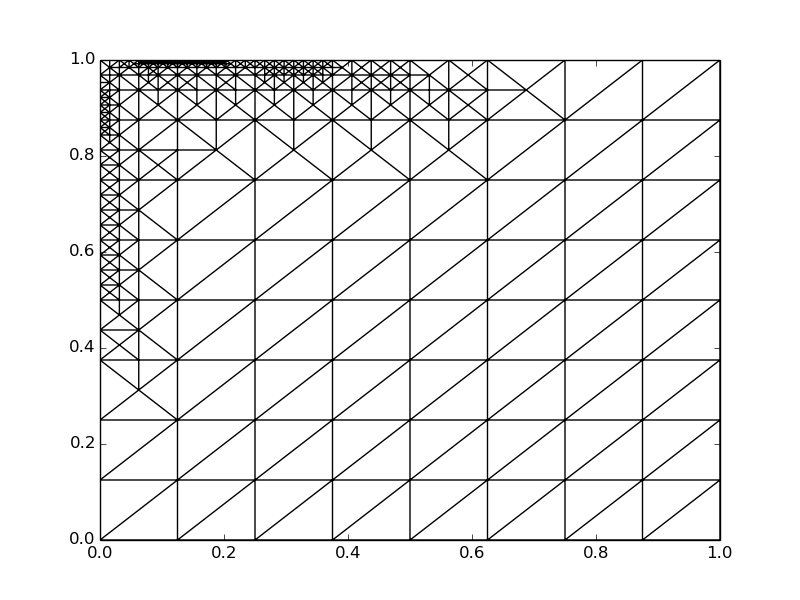} \\
    5 --- 599 cells
  \end{center}
\end{minipage}
\begin{minipage}{0.5\linewidth}
  \begin{center}
    \includegraphics[width=1.0\linewidth] {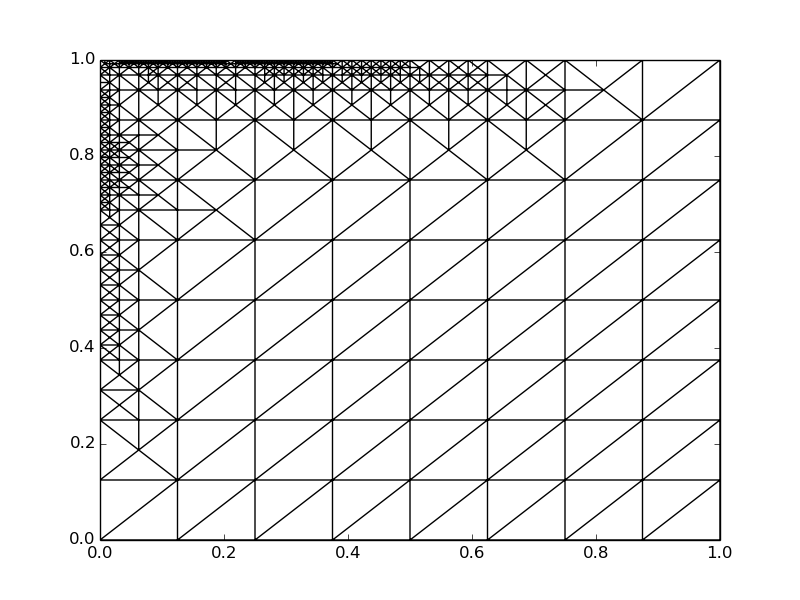} \\
      6 --- 886 cells
  \end{center}
\end{minipage}\hfill
\begin{minipage}{0.5\linewidth}
  \begin{center}
    \includegraphics[width=1.0\linewidth] {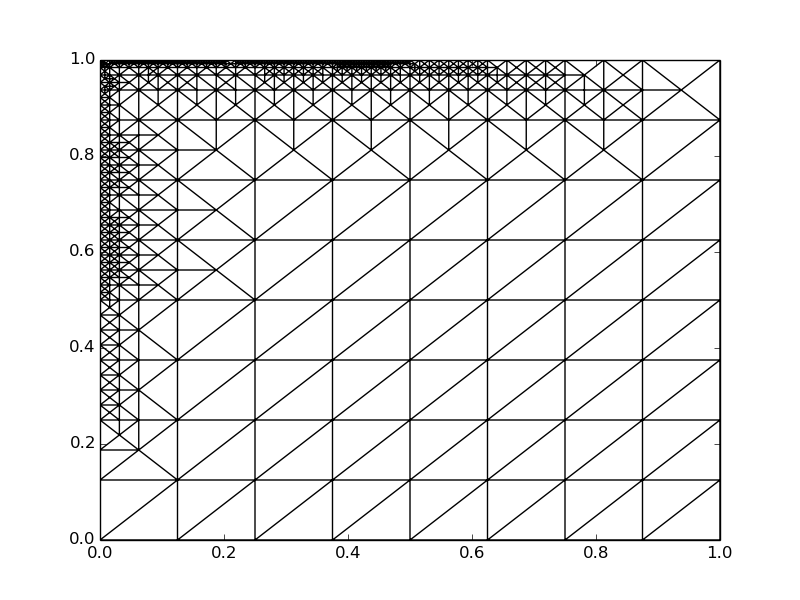} \\
    7 --- 1313 cells
  \end{center}
\end{minipage}
\caption{The grid obtained at succesive steps of adaptation}\label{fig-1}
\end{figure}

\section*{Acknowledgements}
This work was supported by the Russian Foundation for Basic Research  (projects 14-01-00785, 15-01-00026).


\begin{thebibliography}{10}
\providecommand{\url}[1]{\texttt{#1}}
\providecommand{\urlprefix}{URL }

\bibitem{ainsworth2011posteriori}
Ainsworth, M., Oden, J.T.: A Posteriori Error Estimation in Finite Element
  Analysis. Wiley, New York (2000)

\bibitem{baleanu2012fractional}
Baleanu, D.: Fractional Calculus: Models and Numerical Methods. World
  Scientific, New York (2012)

\bibitem{bangerth2013adaptive}
Bangerth, W., Rannacher, R.: Adaptive Finite Element Methods for Differential
  Equations. Birkh{\"a}user, Basel (2003)

\bibitem{brenner2008mathematical}
Brenner, S.C., Scott, L.R.: The mathematical theory of finite element methods.
  Springer, New York (2008)

\bibitem{bueno2012fourier}
Bueno-Orovio, A., Kay, D., Burrage, K.: Fourier spectral methods for
  fractional-in-space reaction-diffusion equations. BIT Numerical Mathematics
  pp. 1--18 (2014)

\bibitem{burrage2012efficient}
Burrage, K., Hale, N., Kay, D.: An efficient implicit fem scheme for
  fractional-in-space reaction-diffusion equations. SIAM Journal on Scientific
  Computing  34(4),  A2145--A2172 (2012)

\bibitem{higham2008functions}
Higham, N.J.: Functions of matrices: theory and computation. SIAM, Philadelphia
  (2008)

\bibitem{ilic2005numerical}
Ilic, M., Liu, F., Turner, I., Anh, V.: Numerical approximation of a
  fractional-in-space diffusion equation, {I}. Fractional Calculus and Applied
  Analysis  8(3),  323--341 (2005)

\bibitem{ilic2006numerical}
Ilic, M., Liu, F., Turner, I., Anh, V.: Numerical approximation of a
  fractional-in-space diffusion equation. {II} with nonhomogeneous boundary
  conditions. Fractional Calculus and applied analysis  9(4),  333--349 (2006)

\bibitem{ilic2009numerical}
Ili{\'c}, M., Turner, I.W., Anh, V.: A numerical solution using an adaptively
  preconditioned lanczos method for a class of linear systems related with the
  fractional poisson equation. International Journal of Stochastic Analysis
  2008,  26 pages (2008)

\bibitem{kilbas2006theory}
Kilbas, A.A., Srivastava, H.M., Trujillo, J.J.: Theory and Applications of
  Fractional Differential Equations. North-Holland mathematics studies,
  Elsevier, Amsterdam (2006)

\bibitem{KnabnerAngermann2003}
Knabner, P., Angermann, L.: Numerical methods for elliptic and parabolic
  partial differential equations. Springer Verlag, New York (2003)

\bibitem{logg2012automated}
Logg, A., Mardal, K.A., Wells, G.: Automated Solution of Differential Equations
  by the Finite Element Method: The FEniCS Book. Springer, Berlin (2012)

\bibitem{miller2012fitted}
Miller, J.J.H., O'Riordan, E., Shishkin, G.I.: Fitted Numerical Methods For
  Singular Perturbation Problems: Error Estimates in the Maximum Norm for
  Linear Problems in One and Two Dimensions. World Scientific, New Jersey

\bibitem{QuarteroniValli1994}
Quarteroni, A., Valli, A.: Numerical Approximation of Partial Differential
  Equations. Springer-Verlag, Berlin (1994)

\bibitem{rognes2013automated}
Rognes, M.E., Logg, A.: Automated goal-oriented error control i: Stationary
  variational problems. SIAM Journal on Scientific Computing  35(3),
  C173--C193 (2013)

\bibitem{roos2008robust}
Roos, H.G., Stynes, M., Tobiska, L.: Robust Numerical Methods for Singularly
  Perturbed Differential Equations: Convection-Diffusion-Reaction and Flow
  Problems. Springer, Berlin (2008)

\bibitem{Thomee2006}
Thom{\'e}e, V.: Galerkin finite element methods for parabolic problems.
  Springer Verlag, Berlin (2006)

\bibitem{vabishchevich2014numerical}
Vabishchevich, P.N.: Numerically solving an equation for fractional powers of
  elliptic operators. Journal of Computational Physics  282(1),  289--302
  (2015)

\bibitem{yagi2009abstract}
Yagi, A.: Abstract Parabolic Evolution Equations and Their Applications.
  Springer, Berlin (2009)

\end{thebibliography}
\end{document}